\documentclass[11pt, reno]{amsart}
\usepackage[parfill]{parskip}
\usepackage[pdftex]{graphicx,color}
\usepackage{amssymb,amsmath,amsthm}
\usepackage{epstopdf}
\usepackage{wrapfig}
\usepackage{marvosym}
\usepackage{tikz}
\usepackage{verbatim}
\usepackage{mathrsfs}
\usepackage{esint}

\numberwithin{equation}{section}

\usepackage[pdftex,bookmarks,colorlinks,breaklinks]{hyperref}
\definecolor{dullmagenta}{rgb}{0.4,0,0.4}   
\definecolor{darkblue}{rgb}{0,0,0.4}
\hypersetup{linkcolor=red,citecolor=blue,filecolor=dullmagenta,urlcolor=darkblue}

\def\dblue#1{\textcolor[rgb]{0,0,0.7}{#1}}

\newcommand{\C}{{\mathbb C}}
\newcommand{\A}{\mathbb{A}}
\newcommand{\R}{{\mathbb R}}
\renewcommand{\H}{\mathbb{H}}
\newcommand{\D}{{\mathbb D}}
\newcommand{\Pol}{\operatorname{Pol}}

\newcommand{\rH}{\rho_{\,\H}}
\newcommand{\rHr}{\rho_{\,\H,r}}
\newcommand{\rC}{\rho_{\,\C}}
\newcommand{\rCg}{\rho_{\,\C,\gamma}}

\newtheorem{thm}{Theorem}[section]
\newtheorem{prop}[thm]{Proposition}
\newtheorem{lemma}[thm]{Lemma}
\newtheorem{cor}[thm]{Corollary}
\newtheorem*{thm-others}{Theorem}
\theoremstyle{remark}
\newtheorem{rem}[thm]{Remark}

\theoremstyle{remark}

\renewcommand{\descriptionlabel}[1]%
         {\dblue{#1:}\\}

\title{Discrepancy densities for planar and hyperbolic Zero Packing}
\author{Aron Wennman}
\address{Department of Mathematics, KTH Royal Institute of Technology,\newline Stockholm, 100 44, Sweden}
\email{aronw@math.kth.se}
\date{\today}
\thanks{The author was supported by the Swedish Research Council, Grant no. 2012-3122, and by the Royal Swedish Academy of Sciences}

\begin{document}

\begin{abstract}
We study the problem of {\em geometric zero packing}, 
recently introduced by Hedenmalm \cite{HedenmalmZeroPacking}.
There are two natural densities associated to this problem:
the {\em discrepancy density} $\rH$, given by
$$
\rH=\liminf_{r\to 1^-}\inf_{f}\frac{\int_{\D(0,r)}\left((1-\lvert z\rvert^2)\lvert f(z)\rvert-1\right)^2\frac{d A(z)}{1-\lvert z\rvert^2}}{\int_{\D(0,r)}\frac{d A(z)}{1-\lvert z\rvert^2}}
$$
which measures the discrepancy in optimal approximation of $(1-\lvert z\rvert^2)^{-1}$ with the modulus of polynomials $f$, 
and it's relative, the {\em tight discrepancy density} $\rH^*$, 
which will trivially satisfy $\rH\leq\rH^*$.
These densities have deep connections to the boundary behaviour of conformal mappings with $k$-quasiconformal extensions, 
which can be seen from the Hedenmalm's result that the universal asymptotic variance $\Sigma^2$ is related to $\rH^*$ by
$\Sigma^2=1-\rH^*$.
Here we prove that in fact $\rH=\rH^*$, resolving a conjecture by Hedenmalm in the positive.
The natural planar analogues $\rC$ and $\rC^*$ to these densities make contact with work of Abrikosov on Bose-Einstein condensates. 
As a second result we prove that also $\rC=\rC^*$.
The methods are based on Ameur, Hedenmalm and Makarov's H{\"o}rmander-type $\bar\partial$-estimates with polynomial growth control. 
As a consequence we obtain sufficiency results on the degrees of approximately optimal polynomials.
\end{abstract}

\maketitle

\section{Introduction}

\subsection{Hyperbolic discrepancy densities}
Let $0<r<1$ and let $f$ be a holomorphic function defined on the unit disk $\D$. 
We shall be concerned with the {\em hyperbolic discrepancy function} $\Phi_f(z,r)$, defined by
$$
\Phi_f(z,r)=\left((1-\lvert z\rvert^2)\lvert f(z)\rvert -1_{\D(0,r)}(z)\right)^2,\qquad z\in\D.
$$
The intuition is that $\Phi_f$ measures the discrepancy between $f$ and 
the hyperbolic metric $\vartheta(z)=(1-\lvert z\rvert^2)^{-1}$.
Since $f$ is holomorphic, $\Delta\log \lvert f(z)\rvert$ is, considered as a distribution, 
a sum of point masses, while $\Delta\log\vartheta(z)$ is a smooth positive density.
This constitutes a clear obstruction to obtain a perfect approximation with holomorphic $f$. 
The term {\em zero packing}, introduced by Hedenmalm \cite{HedenmalmZeroPacking}, 
comes from the realization that this problem can be phrased in terms of optimally discretizing 
the smooth positive mass $\Delta\log\vartheta$ as a sum of point masses -- 
corresponding to the zeros of the holomorphic function $f$.

Our main interest lies in the {\em hyperbolic discrepancy density} $\rH$, and in a related object called 
the {\em tight hyperbolic discrepancy density} $\rH^*$. Without further delay we proceed to define these.
For polynomials $f$ we consider the functionals
$$
\rHr(f)=\frac{\int_{\D(0,r)}\Phi_f(z,r)\frac{dA(z)}{1-\lvert z\rvert^2}}{\int_{\D(0,r)}\frac{dA(z)}{1-\lvert z\rvert^2}}=\frac{\int_{\D(0,r)}\Phi_f(z,r)\frac{dA(z)}{1-\lvert z\rvert^2}}{\log\frac{1}{1-r^2}}
$$
and
$$
\rHr^*(f)=\frac{\int_{\D}\Phi_f(z,r)\frac{dA(z)}{1-\lvert z\rvert^2}}{\int_{\D(0,r)}\frac{dA(z)}{1-\lvert z\rvert^2}}=\frac{\int_{\D}\Phi_f(z,r)\frac{dA(z)}{1-\lvert z\rvert^2}}{\log\frac{1}{1-r^2}}.
$$
In terms of these, the two densities are obtained as
\begin{equation}\label{eqn:rhoH}
\rH=\liminf_{r\to 1^-}\inf_{f}\rHr(f),
\end{equation}
and
\begin{equation}\label{eqn:rhoH-star}
\rH^*=\liminf_{r\to 1^-}\inf_f\rHr^*(f),
\end{equation}
where in both cases the infimum is taken over the set of all polynomials $\Pol(\C)$.

The exact values of these are unknown. The only available quantitative result is due to Hedenmalm, 
which in particular shows that the indicated obstacle to perfect approximation is real, in the sense that $\rH>0$.
\begin{thm}[Hedenmalm, \cite{HedenmalmZeroPacking}]\label{thm:rho-positive}
The hyperbolic discrepancy densities enjoy the estimate
$$
2\times 10^{-8}\leq \rH\leq\rH^*\leq 0.12087.
$$
\end{thm}
For an illustration of the importance of this theorem, in particular of the property that $\rH^*>0$, see Subsection~\ref{ss:q-conf}. 

That the densities satisfy the inequality $\rH\leq\rH^*$ is immediate. 
The density $\rH^*$ differs from $\rH$ in that it adds an $L^2$-punishment near the boundary:
$$
\rHr^*(f)=\rHr(f)+\frac{1}{\log\frac{1}{1-r^2}}\int_{\D\setminus \D(0,r)}\lvert f(z)\rvert^2(1-\lvert z\rvert^2)dA(z),
$$
and in light of this, it is clear that $\rH\leq\rH^*$. 
Hedenmalm has conjectured that equality holds, which is what our main theorem concerns.
\begin{thm}\label{thm:main-hyperbolic}
It holds that $\rH=\rH^*$.
\end{thm}
In the process we obtain the following corollary, which gives a sufficiency result regarding the degree of approximately optimal polynomials. We let $\lceil x\rceil$ denote the smallest integer $n$ with $n\geq x$.
\begin{cor}\label{cor:degree-hyperbolic}
The densities $\rH$ and $\rH^*$ may as well be calculated as
$$
\rH=\liminf_{r\to 1^-}\inf_{f\in\Pol_{n(r)}}\rHr(f), \qquad \rH^*=\liminf_{r\to 1^-}\inf_{f\in \Pol_{n(r)}}\rHr^*(f),
$$
where $n(r)=\left\lceil\frac{r^2}{1-r^2}\right\rceil$.
\end{cor}
Note that $r^2/(1-r^2)$ is the hyperbolic area of the disc $\D(0,r)$.
Ideally, one would want to show that the $n(r)$ zeros of approximating polynomials are 
uniformly spread out with respect to the hyperbolic metric. This, however, remains out of reach at present.

The proof of Theorem~\ref{thm:main-hyperbolic} follows the route suggested by Hedenmalm in \cite{HedenmalmZeroPacking}. We employ the machinery of H{\"o}rmander-type $\bar\partial$-estimates 
with polynomial growth control developed by Ameur, Hedenmalm and Makarov in \cite{AHMBerezin}, 
and an array of variational arguments. The difficulty is to control the size of minimizers of $\rHr(f)$ near the boundary. 
The key ingredient in the solution to this problem is the $L^2$-non-concentration estimate of Theorem~\ref{thm:L2-non-conentration}, which asserts that for minimizers $f$, we have an estimate
$$
\int_{\A((1-\delta)r,r)}\lvert f(z)\rvert^2(1-\lvert z\rvert^2) \,dA(z)=o(1),
$$ 
along certain sequences of radii $r\to 1^-$.

\subsection{Quasiconformal mappings: The integral means spectrum and Quasicircles}\label{ss:q-conf}
The number $\rH^*$ has turned out to play a significant role in the theory of quasiconformal mappings, 
due to it's relation to the universal asymptotic variance $\Sigma^2$.
The number $\Sigma^2$ was introduced in \cite{AIPP}, and is defined in terms of McMullen's asymptotic variance \cite{McMullenVariance}
$$
\sigma^2(g)=\limsup_{r\to 1^-}\frac{\int_{\partial\D}\lvert g(r\zeta)\rvert^2d \sigma(\zeta)}{\log\frac{1}{1-r^2}},
$$
by
$$
\Sigma^2=\sup\{\sigma^2(g): g={\bf P}\mu,\, \lVert \mu\rVert_{L^{\infty}(\D)}= 1\}
$$
where ${\bf P}$ denotes the Bergman projection
$$
{\bf P}f(z)=\int_{\D}\frac{f(w)}{(1-z\bar w)^2}dA(w),\qquad f\in L^1(\D).
$$
For details we refer to e.g. \cite{AIPP, NoQuasiCircles, HedenmalmZeroPacking}. 
Here we mention a couple of recent developments: A well-known conjecture by Prause and Smirnov \cite{PrauseSmirnov} (see also \cite{Jones}) 
for the quasiconformal integral means spectrum $B(k,t)$ stated that
$$
B(k,t)=\begin{cases} \frac{1}{4}k^2\lvert t\rvert^2, & \lvert t\rvert\leq \frac{2}{k}\\ k\lvert t\rvert-1, &\lvert t\rvert>\frac{2}{k}.\end{cases}
$$
Ivrii recently proved \cite{NoQuasiCircles} that $B(k,t)$ satisfies
$B(k,t)\sim\frac{1}{4}\Sigma^2k^2\lvert t\rvert^2$ in the sense that 
$$
\lim_{k\to 0^+}\lim_{t\to0}\frac{B(k,t)}{k^2\lvert t\rvert^2}=\frac{\Sigma^2}{4}
$$
In \cite{AIPP}, Astala, Ivrii, Per{\"a}l{\"a} and Prause obtained the bounds $0.879\leq \Sigma^2\leq 1$, 
and it was conjectured that $\Sigma^2=1$, which would be implied by the above conjecture.
However, in addition to Theorem~\ref{thm:rho-positive}, 
Hedenmalm \cite{HedenmalmZeroPacking} has recently proven that
\begin{equation}\label{eqn:Hedenmalmrho-star}
\Sigma^2=1-\rH^*,
\end{equation}
and taken together, these facts refute the conjecture.

The same family of objects is also relevant to work by Ivrii on the dimension of $k$-quasicircles: 
If $D(k)$ is the maximal Hausdorff dimension of a $k$-quasicircle, a theorem of Smirnov (see the book \cite{Astala} for an exposition) says that
\begin{equation}\label{eqn:Smirnov-bound}
D(k)\leq 1+k^2,\qquad 0\leq k< 1.
\end{equation}
Astala conjectured this result \cite{AstalaAreaDist}, and furthermore suggested that this bound is sharp. Ivrii \cite{NoQuasiCircles} proved that
\begin{equation}\label{eqn:Ivrii-quasicircle}
D(k)=1+\Sigma^2k^2+O(k^{8/3-\epsilon}),
\end{equation}
which together with Theorem~\ref{thm:rho-positive} and \eqref{eqn:Hedenmalmrho-star} effectively disproves the latter part of the conjecture.

\subsection{The planar discrepancy densities}
We are also interested in planar analogues of the densities $\rH$ and $\rH^*$. 
For $R>0$ and an entire function $f(z)$, we consider the {\em planar discrepancy function}
$$
\Psi_f(z,R)=\left(\lvert f(z)\rvert e^{-\lvert z\rvert^2}-1_{\D(0,R)}(z)\right)^2,
$$
and set 
\begin{equation}\label{eqn:rhoC}
\rC=\liminf_{R\to\infty}\inf_f\frac{\int_{\D(0,R)}\Psi_f(z,R)dA(z)}{\int_{\D(0,R)}dA(z)}=\liminf_{R\to\infty}\inf_f\frac{1}{R^2}\int_{\D(0,R)}\Psi_f(z,R)dA(z),
\end{equation}
and correspondingly
\begin{equation}\label{eqn:rhoC-star}
\rC^*=\liminf_{R\to\infty}\inf_f\frac{\int_{\C}\Psi_f(z,R)dA(z)}{\int_{\D(0,R)}dA(z)}=\liminf_{R\to\infty}\inf_f\frac{1}{R^2}\int_{\C}\Psi_f(z,R)dA(z),
\end{equation}
where the infimum is taken over all polynomials. The next result corresponds completely to Theorem~\ref{thm:main-hyperbolic}.
\begin{thm}\label{thm:main-planar}
It holds that $\rC=\rC^*$.
\end{thm}
Also analogously to the hyperbolic setting, we may say something about the degree of approximately minimal polynomials.
\begin{cor}\label{cor:degree-planar}
The densities $\rC$ and $\rC^*$ are unchanged if the infimum in \eqref{eqn:rhoC} and \eqref{eqn:rhoC-star} are taken over $\Pol_{n(R)}(\C)$ instead of over $\Pol(\C)$, where 
$$
n(R)=\left\lceil2R^2\right\rceil.
$$
\end{cor}
It thus suffices to consider polynomials of degrees that are essentially proportional to the area of the disk $\D(0,R)$. Here the factor 2 is natural, since each zero carries a mass of $1/2$.

By a change of variables, we may perform the calculation
$$
\frac{1}{R^2}\int_{\D(0,R)}\Psi_f(z,R)dA(z)=\int_{\D}\left(\lvert f(Rw)\rvert e^{-R^2\lvert w\rvert^2}-1\right)^2dA(w).
$$
Since the dilation $f\mapsto f_R$ where $f_R(z)=f(Rz)$ will not affect holomorphicity of $f$, we may as well use the functionals
\begin{align}
\label{eq:rg}\rCg(f)&=\int_{\D}\left(\lvert f(z)\rvert e^{-\gamma\lvert z\rvert^2}-1\right)^2dA(z), \\
\label{eq:rg-star}\rCg^*(f)&=\int_{\C}\left(\lvert f(z)\rvert e^{-\gamma\lvert z\rvert^2}-1_{\D}(z)\right)^2dA(z),
\end{align}
and instead obtain the densities by
$$
\rC=\liminf_{\gamma\to\infty}\inf_f\rCg(f),\qquad \rC^*=\liminf_{\gamma\to\infty}\inf_f\rCg^*(f).
$$
For the purpose of this paper, it turns out to be more convenient to work with this formulation. 

Theorem~\ref{thm:main-planar} could be seen as a toy problem for Theorem~\ref{thm:main-hyperbolic}, 
in that the $\bar\partial$-estimates are slightly more readily applicable in this setting.
Since the proof of this illustrates the methods used very transparently, we present it first.

\subsection{Relation to Bose-Einstein Condensates}
The planar density $\rC$ is part of a bigger family of densitites, $\rho^{\langle \beta\rangle}_{\C}$, defined for $\beta>0$ by
$$
\rho^{\langle \beta\rangle}_{\C}=\liminf_{R\to\infty}\inf_{f}\int_{\D(0,R)}\left(\lvert f(z)\rvert^\beta e^{-\lvert z\rvert^2}-1\right)^2dA(z).
$$
The case $\beta=1$ is the density $\rC$. Also of particular interest is the case $\beta=2$, which can be traced back 
to work by Abrikosov \cite{Abrikosov} on Bose-Einstein Condensates. 
Abrikosov suggested that it should be enough to look for minimizers among functions quasiperiodic with respect to lattices. 
The conjecture, which is attributed to Abrikosov in \cite{HedenmalmZeroPacking}, is that the equilateral triangular 
lattice should be the correct choice for any $\beta$. 

Consider the triangular lattices $2\omega_1\mathbb{Z}+2\omega_2\mathbb{Z}$, where $\omega_1\in\R^+$ and $\omega_2=\omega_1 e^{i\theta}$. 
For each lattice, good candidates $f_0=f_{0,\beta,\theta}$ for minimizers of $\rho^{\langle \beta\rangle}_{\C}(f)$ are 
given explicitly in terms of Weierstrass' $\sigma$-function, see \cite{HedenmalmZeroPacking}.
A numerical computation using this choice yields the value
$$
\rho^{\langle 1\rangle}_{\C}(f_0)=\lim_{R\to\infty}\frac{1}{R^2}\int_{\D(0,R)}\left(\lvert f_0(z)\rvert e^{-\lvert z\rvert^2}-1\right)^2dA(z)=0.061203\ldots,
$$
which in particular gives a numerical bound $\rC\leq 0.061203$. 
\begin{figure}[h]
    \centering
    \includegraphics[width=0.6\textwidth]{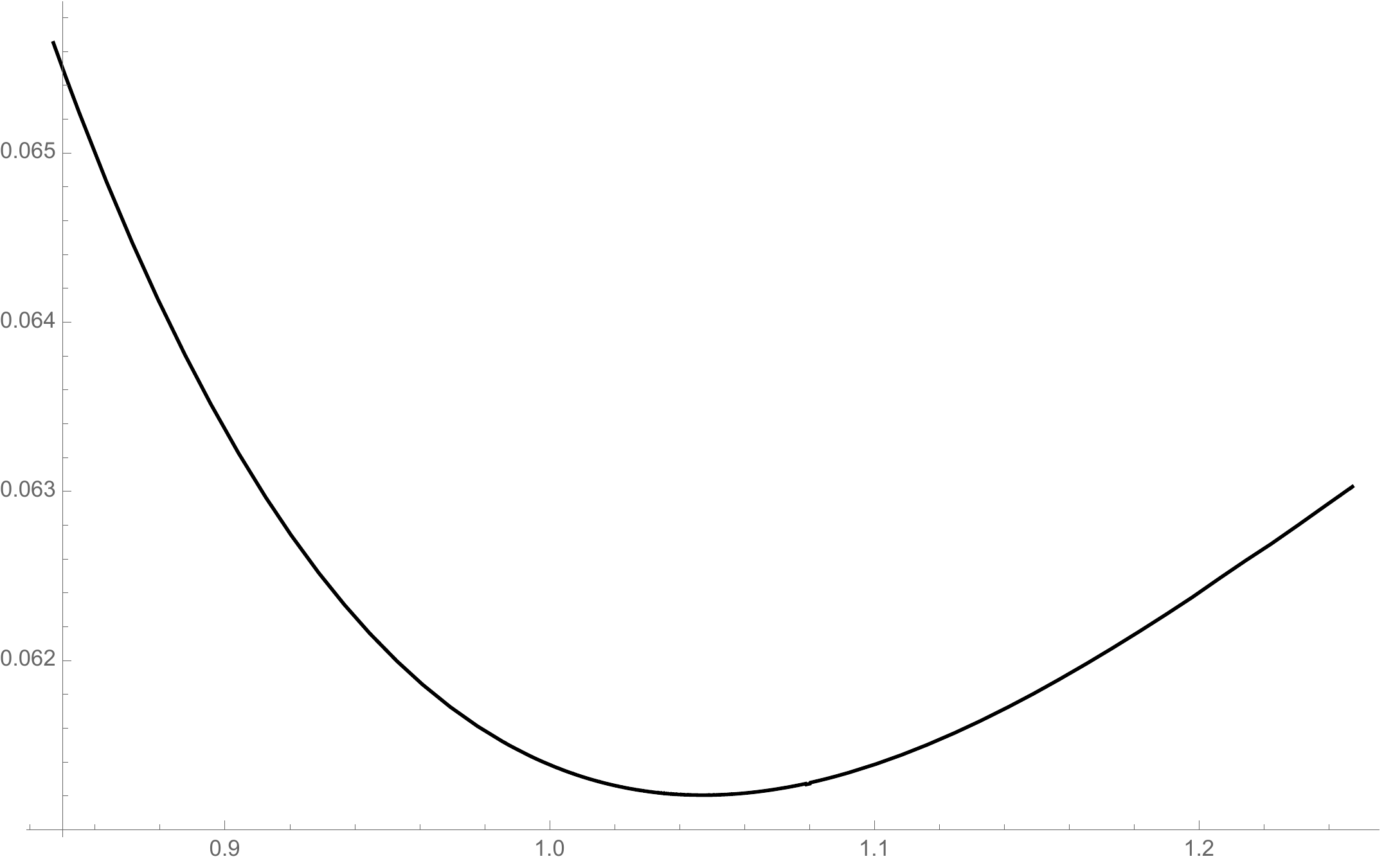}
    \caption{$\rho^{\langle 1\rangle}_{\C}(f_{0,\theta})$ for $\theta$ near $\pi/3$, plotted using {\textsc{Mathematica}}}
    \label{fig:rho-theta}
\end{figure}

In Figure~\ref{fig:rho-theta}, we have plotted $\rho^{\langle 1\rangle}_{\C}(f_{0,\theta})$ for triangular lattices with different angles. 
The minimum appears to be at $\theta=\pi/3$, in support of the conjecture that the equilateral triangular lattice is optimal.

\subsection{Notation and special conventions}
By $\D(z_0, r)$ we mean the open disk centred at $z_0\in \C$ with radius $r>0$, 
and by $\A(z_0, r,R)$ we mean an open annulus $\D(z_0,R)\setminus \overline{\D(z_0,r)}$, where $R>r>0$. 
When $z_0=0$ we simply denote the annulus by $\A(r,R)$.

By $dA$ we mean the normalized area measure,
$$
dA(z)=\frac{dxdy}{\pi},\qquad z=x+iy\in\C.
$$
We shall make frequent use of the Cuachy-Riemann operators
$$
\partial =\frac12\left(\frac{\partial}{\partial x}-i\frac{\partial}{\partial y}\right), \qquad  \bar\partial =\frac12\left(\frac{\partial}{\partial x}+i\frac{\partial}{\partial y}\right).
$$
We use the Laplacian $\Delta$, which is a quarter of the usual Laplacian, where this normalization is chosen so that it factorizes as 
$$
\Delta=\partial\bar\partial.
$$
We will frequently consider $\bar\partial$-equations of the kind
\begin{equation}\label{eqn:d-bar}
\bar\partial  u= f,\qquad f\in L^2(\mathscr{T}, d\mu),
\end{equation}
where $\mathscr{T}$ is some compact subset of $\C$ and $\mu$ is a measure. By a solution $u$ to \eqref{eqn:d-bar}, 
we mean an element $u\in W^{1,2}(d \mu, \mathscr{T})$ such that $\bar\partial u=f$ in $L^2(d\mu, \mathscr{T})$.

\subsection{Acknowledgements} 
I want to thank H{\aa}kan Hedenmalm for his generousity with his insights, 
for useful comments while preparing the manuscript and not least for suggesting the study of this problem.
I'm greatful to Oleg Ivrii for numerous suggestions on the manuscript, and for him sharing his insights about intriguing remaining questions.
Ivrii \cite{PrivateIvrii} has suggested another approach to this problem, based on entirely different methods, which I hope will be further explored.
Thanks also to Simon Larson for inspiring discussions and for proof-reading.

\section{Preliminaries}
\subsection{Function spaces with polynomial growth}
By $\Pol_n(\C)$ we mean the space of polynomials of degree at most $n-1$, 
and we denote by $\Pol(\C)$ the space of all polynomials.

Let $\phi$ denote a real-valued function defined on some domain $\Omega$, 
possibly the entire plane. 
By $L^2_\phi=L^2_\phi(\Omega)$ we mean the usual $L^2$-space with inner product
$$
\left\langle f,g\right\rangle_\phi:=\int_{\Omega}f(z)\overline{g(z)}e^{-\phi(z)}dA(z).
$$
We denote by $A(\Omega)$ the set of holomorphic functions on $\Omega$, 
and let $A^2_\phi=A^2_\phi(\Omega)$ be the intersection 
$$
A^2_\phi=A(\Omega)\cap L^2_\phi,
$$
endowed with the inner product inherited from $L^2_\phi$.

When $\Omega=\C$, we also consider the spaces
$$
L^2_{n,\phi}=\left\{f\in L^2_\phi : \lvert f(z)\rvert \leq \lvert z\rvert^{n-1}+O(1) \text{\;as\;} z\to\infty\right\},
$$
with a polynomial growth restriction at infinity, and the space
$$
A^2_{n,\phi}=L^2_\phi(\C)\cap\Pol_n(\C)=A^2_\phi(\C)\cap L^2_{n,\phi}(\C).
$$
We will be especially concerned with the spaces $A^2_{\phi}$ and $A^2_{n,\phi}$ in the 
cases when $\phi=\log\frac{1}{1-\lvert z\rvert^2}$ and $\phi=c\lvert z\rvert^2$ for some constant $c>0$.
In the literature these are often referred to as (polynomial) Bergman and Fock spaces, respectively.

\subsection{Cut-off functions} \label{ss:cut-off}
We will find the need to make use of cut-off functions $\chi=\chi_{\delta,r}$ that are 
identically one on a disk $\D(0,r(1-\delta))$, and vanish off the slightly bigger disk $\D(0,r)$. 
These can be chosen so as to satisfy the estimates
$$
\lvert \bar\partial \chi_{\delta,r}\rvert^2\leq \frac{C}{\delta^2r^2} 1_{\A((1-\delta)r,r)},\qquad \lVert \bar\partial\chi_{\delta,r}\rVert_{L^2}^2\leq \frac{4}{\delta},
$$
for $0<r<1$ and $0<\delta<1$.
An example of such a function is given by
\begin{equation}\label{eqn:cut-off}
\chi_{\delta,r}(z)=
\begin{cases}
1, & \lvert z\rvert \leq (1-\delta)r \\
\left(\frac{1}{\delta}-\frac{\lvert z\rvert}{\delta r}\right)^2, & (1-\delta)r< \lvert z\rvert\leq r \\
0, & \lvert z\rvert >r
\end{cases}
\end{equation}
Note that this function is merely Lipschitz. In case one requires more regularity, 
it suffices to note that the above properties should be stable under smoothing procedures, such as convolution.

\subsection{A $\bar\partial$-estimate with polynomial growth control}
We rely on methods from the work of Ameur, Hedenmalm and Makarov \cite{AHMBerezin}. 
They prove a version of H{\"o}rmander's classical $\bar\partial$-estimates, which gives polynomial growth control at infinity.
Here we only need the following direct special case of \cite[Theorem~4.1]{AHMBerezin}.

\begin{thm}\label{thm:AHMBerezin}
Let $\mathscr{T}$ be a compact subset of $\C$, 
and denote by $\phi, \widehat\phi$ two real-valued functions on $\C$ of class $\mathscr{C}^{1,1}$, such that
\begin{itemize}
\item $\phi(z)= \widehat\phi(z)$ for $z\in\mathscr{T}$ and $\widehat\phi(z)\leq \phi(z)$ for $z\in\C$,
\item $\Delta\widehat\phi>0$ on $\mathscr{T}$, and $\Delta\widehat\phi\geq 0$ on $\C$.
\item $\widehat\phi(z)=\tau\log \lvert z\rvert^2+O(1)$ as $z\to\infty$.
\end{itemize}
Then, for any integer $n\geq \tau$ and $f\in L^\infty(\mathscr{T})$, the $L^2_{n,\phi}$-minimal solution $u_{0,n}$ to $\bar\partial u=f$ exists and satisfies
$$
\int_\C\lvert u_{0,n}\rvert^2 e^{-\phi}dA\leq \int_{\mathscr{T}}\lvert f\rvert^2 \frac{e^{-\phi}}{\Delta\widehat\phi}dA.
$$
\end{thm}

\begin{rem} a)
We remark that we may allow the function $\phi$ to take on the value $+\infty$ on $\mathscr{T}^c$. That this is the case can be seen by applying the theorem to the pair $(\widehat\phi,\widehat\phi)$ to obtain that the $L^2_{n,\widehat\phi}$-minimal solution $v_0$ to $\bar\partial v=f$ satisfies
$$
\int_\C\lvert v_0\rvert^2 e^{-\widehat\phi}dA\leq \int_\mathscr{T}\lvert f\rvert^2\frac{e^{-\phi}}{\Delta\phi}dA.
$$
Since $L^2_{n, \widehat\phi}\subset L^2_{n, \phi}$, it follows that $\int \lvert u_0\rvert^2e^{-\phi}dA\leq \int \lvert v_0\rvert^2e^{-\phi}dA$. We may therefore infer the desired result from the fact that $\widehat\phi\leq \phi$.

b) The original theorem pertains to a wider class of $\widehat\phi$ in terms of growth at infinity than considered here, but requires that $\Delta\widehat\phi>0$ on the entire plane. For the specific form of $\widehat{\phi}$ considered here, this requirement may easily be removed by approximation, as is done in \cite[Section~4.4]{AHMBerezin}. Indeed, by letting
$$
\widehat\phi_\epsilon(z)=\left(1-\frac{\epsilon}{\tau}\right)\widehat\phi(z)+\epsilon\log(1+\lvert z\rvert^2)
$$
for $\epsilon>0$ and applying the theorem to $(\phi,\widehat\phi_\epsilon)$, the desired inequality follows by letting $\epsilon\to0$.
\end{rem}

\section{The Planar Case: Proof of Theorem~\ref{thm:main-planar}}
\subsection{The fundamental $\bar\partial$-estimate}
Recall the functionals $\rCg(f)$ and $\rCg^*(f)$ from \eqref{eq:rg} and \eqref{eq:rg-star}. To prove that $\rC=\rC^*$, we follow the approach suggested in \cite{HedenmalmZeroPacking}, which is to modify minimizers $f$ of $\rCg(f)$ outside $\D$ so as to make sure that $\rCg^*(f)$ is close to $\rCg(f)$.
This is done in two steps: first one multiplies $f$ by a cut-off function $\chi$ that vanishes outside $\D$. 
Secondly, one must correct $\chi f$ so that it once again becomes a polynomial. 
This is done using the H{\"o}rmander-type $\bar\partial$-techniques of \cite{AHMBerezin}, i.e. Theorem~\ref{thm:AHMBerezin} above.

Denote by $\chi=\chi_{\delta, 1}$ the cut-off function from Subsection~\ref{ss:cut-off}. 
We will apply Theorem~\ref{thm:AHMBerezin} to the equation $\bar\partial u=\bar\partial(\chi f)=f\bar\partial\chi$, where $f$ is a minimizer of $\rCg(f)$. 
One observes that $u$ is then the desired correction: $\bar\partial(\chi f-u)=f\bar\partial \chi-u=0$, so $\chi f-u$ is holomorphic. 
We let $\mathscr{T}=\D$, $\phi(z)=2\gamma\lvert z\rvert^2$ and we define $\widehat\phi$ to be the unique function satisfying
\begin{itemize}
\item $\widehat{\phi}=\phi$ on $\D$, 
\item $\widehat\phi\in\mathscr{C}^{1,1}$ and $\Delta\phi\geq 0$.
\item $\widehat\phi$ is pointwise minimal with these conditions satisfied.
\end{itemize}
Since $\phi$ is radial, this reversed obstacle problem is easy, 
and we can give an explicit formula for $\widehat\phi$:
$$
\widehat\phi(z)=\begin{cases}2\gamma\lvert z\rvert^2, &\lvert z\rvert <1 \\ 2\gamma\log\lvert z\rvert^2+2\gamma, & \lvert z\rvert\geq 1.\end{cases}
$$
To verify that this formula is correct, we note that $\widehat\phi=\phi$ on $\D$ and 
$\partial_n\widehat\phi=\partial_n\phi$ on $\partial\D$, so since $\widehat{\phi}$ is smooth away from $\partial\D$, it inherits the $\mathscr{C}^{1,1}$-regularity from $\phi$. 
That $\widehat\phi$ is harmonic for $\lvert z\rvert>1$ ensures minimality by use of the maximum principle.
It is also easy to verify that $\widehat\phi\leq\phi$.

Thus all conditions of Theorem~\ref{thm:AHMBerezin} are satisfied,
and we may infer that the $L^2_{n,\phi}$-minimal solution $u$ to the $\bar\partial$-equation satisfies
$$
\int_\C\lvert u(z)\rvert^2e^{-2\gamma\lvert z\rvert^2}dA(z)\leq \int_{\mathscr{T}}\lvert \bar\partial(\chi f)\rvert^2 \frac{e^{-\phi}}{\Delta\phi}dA(z)=\frac{1}{2\gamma}\int_{\A(1-\delta, 1)}\lvert f(z)\rvert^2\lvert \bar\partial\chi(z)\rvert^2e^{-2\gamma\lvert z\rvert^2}dA(z).
$$
We have thus arrived at the following result, which controls the $L^2$-norm of the correction to non-holomorphicity of $\chi f$.

\begin{thm}\label{thm:d-bar-planar}
Let $\gamma>0$ and let $f$ be a bounded holomorphic function on the unit disk. 
Then there exists a solution $u$ to $\bar\partial u=\bar\partial(\chi f)$ that satisfies the estimate
$$
\int_{\C}\lvert u(z)\rvert^2e^{-2\gamma\lvert z\rvert^2}d A(z)\leq \frac{1}{2\gamma}\int_{\A((1-\delta),1)}\lvert f(z)\rvert^2\,\lvert\bar\partial \chi(z)\rvert^2e^{-2\gamma\lvert z\rvert^2}d A(z),
$$
with polynomial growth control
$$
\lvert u(z)\rvert = O\left(\lvert z\rvert^{n-1}\right),\qquad z\to\infty
$$
where $n=n(\gamma)=\lceil 2\gamma\rceil$.
\end{thm}

\subsection{Existence and a priori control of minimizers}
We first obeseve that for a fixed $\gamma<\infty$, there exists a holomorphic function $f_0$ on $\D$ that minimizes $\rCg(f)$. 
Indeed, let $f_n\in\Pol_n(\C)$ be a sequence of polynomials, for which $\rCg(f_n)\to\inf\rCg(f)$. 
We may assume that they are abolute minimizers within their respective spaces $Pol_n(\C)$. Let $\phi=2\gamma\lvert z\rvert^2$. 
A simple variational argument, see the proof of Lemma~\ref{lem:unif-norm-planar} below, shows that the $L^2$-norms 
$\lVert f_n\rVert_{A^2_\phi(\D)}$ are uniformly bounded.
Denote by $K_\phi(z,w)$ the reproducing kernel for the space $A^2_\phi(\D)$. 
By Cauchy--Schwarz inequality, one finds the pointwise bound
$$
\lvert f(z)\rvert^2\leq \lVert K_\phi(\,\cdot\,, z)\rVert^2_\phi \,\lVert f\rVert_{A^2_\phi}^2, \qquad z\in\D
$$
for $f\in A^2_\phi$, which yields a uniform bound $\lVert f_{n\, \vert K}\rVert_{\infty}\leq M_K$ independently of $n$, 
for each fixed compact subset $K$ of $\D$.
By a normal families argument, there exists a holomorphic function $f_0$ and a subsequence $\{n_k\}$ along which $f_{n_k}\to f_0$ uniformly on compact subsets. By Fatou's Lemma we find that
$$
\rCg(f_0)\leq \liminf_{k\to\infty}\rCg(f_{n_k}),
$$
so $f_0$ is indeed a minimizer.

These minimizers turn out to have good properties, even uniformly in the parameter $\gamma$.
\begin{lemma}\label{lem:unif-norm-planar}
Assume that $f=f_\gamma$ is a minimizer of $\rCg(f)$. Then
$$
\int_{\D}\lvert f(z)\rvert^2e^{-2\gamma\lvert z\rvert^2}dA(z)=\int_{\D}\lvert f(z)\rvert e^{-\gamma\lvert z\rvert^2}dA(z),
$$  
and both expressions are bounded as $\gamma\to\infty$.
\end{lemma}
\begin{proof}
Define $V(\alpha)$ for $\alpha>0$ by
$$
V(\alpha)=\rCg(\alpha f)
$$
Since $\alpha f$ is holomorphic whenever $f$ is, 
it is clear that we may vary $\alpha$ within the class of admissible functions for the infimum. It follows that
$V'(1)=0$, which after expanding the square reads
$$
\int_{\D}\lvert f(z)\rvert^2e^{-2\gamma\lvert z\rvert^2}dA(z)=\int_{\D}\lvert f(z)\rvert e^{-\gamma\lvert z\rvert^2}dA(z),
$$
which is exactly the first assertion. Using this property, one finds that
$$
\rCg(f)=1-\int_{\D}\lvert f(z)\rvert e^{-\gamma\lvert z\rvert^2}dA(z)
$$
and since $\rCg(f)>0$, this implies the boundedness assertion.
\end{proof}
The next results controls the $L^1$-norm of minimizers near $\partial\D$, 
and will be referred to as the {\em $L^1$-non-concentration estimate}.
\begin{lemma}\label{lemma:L1-non-concentration-planar}
Assume that $\delta=\delta_\gamma=o(1)$ as $\gamma\to\infty$. 
Let $\{f_\gamma\}$ be a sequence of minimizers of $\rCg(f)$. Then
$$
\int_{\A(1-\delta, 1)}\lvert f(z)\rvert e^{-\gamma\lvert z\rvert^2}d A(z)=o(1), \qquad \gamma\to\infty.
$$
\end{lemma}
\begin{proof}
By Cauchy--Schwarz inequality, 
$$
\int_{\A(1-\delta, 1)}\lvert f(z)\rvert e^{-\gamma\lvert z\rvert^2}d A(z)\leq \left(\int_{\A(1-\delta, 1)} \lvert f(z)\rvert^2e^{-2\gamma\lvert z\rvert^2}dA(z)\right)^{1/2}\left(\int_{\A(1-\delta,1)}dA(z)\right)^{1/2}
$$
The first integral is uniformly bounded as $n\to\infty$ by Lemma~\ref{lem:unif-norm-planar}, 
and the area of annuli with radii $(1-\delta, 1)$ tends to zero as $\delta\to 0$.
\end{proof}
It turns out to be beneficial to introduce one more parameter in the functionals. For $\alpha>0$ we consider
$$
\rho_{\,\C,\gamma,\alpha}(f)=\int_{\D}\left(\lvert f(z)\rvert e^{-\alpha\gamma \lvert z\rvert^2}-1\right)^2dA(z).
$$
\begin{prop}\label{prop:parameter-planar}
For any sequence $\alpha_\gamma\to 1$ it holds that
$$
\liminf_{\gamma\to\infty}\inf_f\rho_{\,\C,\gamma,\alpha}(f)=\rC.
$$
\end{prop}
\begin{proof}
This is immediate after the change of variables $w=\alpha^{1/2}\gamma^{1/2}z$, 
by the fact $\Pol(\C)$ is invariant under dilations and the original definition \eqref{eqn:rhoC} of $\rC$. 
\end{proof}

\subsection{An $L^2$-non-concentration estimate}
The point of this section is to control the growth of minimizers $f_\gamma$ near $\partial\D$. 
This is done effectively via the following theorem, which we will refer to as the planar $L^2$-non-concentration estimate.
\begin{thm}\label{thm:L2-non-conc-planar}
Let $f$ be a minimizer of $\rCg(f)$ and let $\delta =\delta_\gamma\to0$. Then
$$
\int_{\A(1-\delta, 1)}\lvert f(z)\rvert^2 e^{-2\gamma\lvert z\rvert^2}d A(z)=o(1),
$$
as $\gamma\to\infty$ along a subsequence $\Gamma=\{\gamma_k\}$ for which there exist polynomials $f_k$ such that $\rCg(f_k)\to\rC$.
\end{thm}
\begin{proof}
Fix such a sequence $\Gamma$, and let $f_\gamma$ denote a sequence of minimizers. 
Let $\alpha=(1-\delta)^2$ and $g_\gamma(z)=f_\gamma(\alpha^{1/2} z)$. 
Computing $\rho_{\,\C, \gamma, \alpha}(g)$ we find that
\begin{align*}
\rho_{\,\C, \gamma, \alpha}(g)
=&\int_{\D(0,1-\delta)}\lvert f(z)\rvert^2 e^{-2\gamma\lvert z\rvert^2}dA(z)-2\int_{\D(0,1-\delta)}\lvert f(z)\rvert e^{-\gamma\lvert z\rvert^2}d A(z)+1\\
=&\int_{\D}\lvert f(z)\rvert^2 e^{-2\gamma\lvert z\rvert^2}d A(z)-2\int_{\D}\lvert f(z)\rvert e^{-\gamma\lvert z\rvert^2}d A(z)+1\\
&-\int_{\A(1-\delta,1)}\lvert f(z)\rvert^2e^{-2\gamma\lvert z\rvert^2}dA(z)+2\int_{\A(1-\delta,1)}\lvert f(z)\rvert e^{-\gamma\lvert z\rvert^2}dA(z).
\end{align*}
From the $L^1$-non-concentration estimate (Lemma~\ref{lemma:L1-non-concentration-planar}), it follows that we may write
$$
\rho_{\,\C, \gamma, \alpha}(g)=\rho_{\,\C,\gamma}(f)-\int_{\A(1-\delta,1)}\lvert f(z)\rvert^2e^{-2\gamma\lvert z\rvert^2}dA+o(1).
$$
Since the remaining integral is positive;
$$
\liminf_{\gamma\to\infty, \gamma\in\Gamma}\rho_{\,\C, \gamma, \alpha}(g)=\rC-\limsup_{\gamma\to\infty, \gamma\in\Gamma}\int_{\A(1-\delta,1)}\lvert f(z)\rvert^2e^{-2\gamma\lvert z\rvert^2}dA.
$$
Proposition~\ref{prop:parameter-planar} tells us that $\liminf \rho_{\,\C, \gamma, \alpha}(g)\geq \rC$. 
If the above equation is to refrain from violating this, it must hold that
$$
\int_{\A(1-\delta,1)}\lvert f(z)\rvert^2e^{-2\gamma\lvert z\rvert^2}dA=o(1)
$$
which is the desired conclusion.
\end{proof}

\subsection{Proof of Theorem~\ref{thm:main-planar}}
Let $\Gamma$ be a sequence of numbers $\gamma\to\infty$, along which
$$
\lim_{k\to\infty}\inf_f\rho_{\,\C, \gamma_k}(f)=\rC.
$$
We will take all subsequent limits along this sequence.

Let $\delta \geq \gamma^{-1/2}$, so that by \eqref{eqn:cut-off} we have the bound
$$
\lvert \bar\partial \chi_{\delta,1}\rvert^2\leq C\gamma.
$$
For each $\gamma\in\Gamma$, let $f=f_\gamma$ be a minimizer of $\rCg(f)$, 
and let $u$ be a solution to $\bar\partial u=\bar\partial(\chi f)$, as in Theorem~\ref{thm:d-bar-planar}. 
Then
\begin{equation}\label{eqn:exterior-planar}
\int_{\C}\lvert u\rvert^2e^{-2\gamma \lvert z\rvert^2}dA\leq \frac{1}{2\gamma}\int_{\A(1-\delta, 1)}\lvert \bar\partial \chi\rvert^2\lvert f\rvert^2e^{-2\gamma\lvert z\rvert^2}dA\leq C\int_{\A(1-\delta,1)}\lvert f\rvert^2e^{-2\gamma\lvert z\rvert^2}dA=o(1)
\end{equation}
where the asymptotics follows from the $L^2$-non-concentration estimate in Theorem~\ref{thm:L2-non-conc-planar}.

Let $\nu=\chi f-u$. Then $\nu$ is holomorphic, and since $\chi$ has compact support 
and $\lvert u(z)\rvert=O(\lvert z\rvert^{n-1})$, it follows by Liouville's theorem that $\nu\in\Pol_n(\C)$.
We calculate the functional $\rCg^*(\nu)$ as
\begin{equation}\label{eqn:main-planar}
\rCg^*(\nu)=\int_{\C\setminus\D}\lvert u\rvert^2e^{-2\gamma\lvert z\rvert^2}dA+\int_{\D}\lvert \nu\rvert^2e^{-2\gamma\lvert z\rvert^2}dA-2\int_\D\lvert \nu\rvert e^{-\gamma\lvert z\rvert^2}dA+1
\end{equation}
The first term in \eqref{eqn:main-planar} is $o(1)$ by \eqref{eqn:exterior-planar}. 

We turn to the $L^p$-norms of $\nu e^{-\gamma\lvert z\rvert^2}$. 
By the $L^1$-non-concentration estimates, we have that
$$
\int_{\D}\lvert \nu\rvert e^{-\gamma\lvert z\rvert^2}dA=\int_{\D}\lvert f\rvert e^{-\gamma\lvert z\rvert^2}dA+\int_{\D}(\lvert \nu\rvert -\lvert f\rvert) e^{-\gamma\lvert z\rvert^2}dA
$$
and
$$
\left\lvert \int_{\D}(\lvert \nu\rvert -\lvert f\rvert) e^{-\gamma\lvert z\rvert^2}dA\right\rvert \leq 2\int_{\A(1-\delta,1)}\lvert f\rvert e^{-\gamma\lvert z\rvert^2}dA+\int_\D \lvert u\rvert e^{-\gamma\lvert z\rvert^2}dA=o(1),
$$
where the last assertion follows from the $\bar\partial$-estimate \eqref{eqn:exterior-planar} 
and the $L^1$-non-concentration estimate. Thus 
$$
\int_\D \lvert \nu\rvert e^{-\gamma\lvert z\rvert^2}d A=\int_\D \lvert f\rvert e^{-\gamma\lvert z\rvert^2}d A+o(1).
$$
Turning to the $L^2$-norms, 
$$
\int_\D\lvert \nu\rvert^2 e^{-2\gamma\lvert z\rvert^2}dA=\int_{\D}\lvert \chi f\rvert^2e^{-2\gamma\lvert z\rvert^2}dA +\int_{\D}(\lvert u\rvert^2-2\Re(\chi f\overline{u}))e^{-2\gamma\lvert z\rvert^2}dA.
$$
The latter integral is $o(1)$ by the $\bar\partial$-estimate \eqref{eqn:exterior-planar} 
and an application of the Cauchy--Schwarz inequality. 
The former satisfies
$$
\int_{\D}\lvert \chi  f\rvert^2e^{-2\gamma\lvert z\rvert^2}dA=\int_{\D}\lvert f\rvert^2e^{-2\gamma\lvert z\rvert^2}dA+o(1),
$$
in light of Theorem~\ref{thm:L2-non-conc-planar}. It follows from this that
$$
\rCg^*(\nu)=\rCg(f)+o(1), \qquad \gamma\to\infty, \gamma\in\Gamma.
$$
Since $\nu$ are admissible polynomials, this implies that
$$
\rC^*\leq \liminf_{\gamma\to\infty}\rCg^*(\nu)\leq\lim_{\gamma\to\infty, \gamma\in\Gamma}\rCg(f)=\rC.
$$
The reversed inequality is known to hold, so it follows that $\rC=\rC^*$.
\qed

\section{The Hyperbolic Case: Proof of Theorem~\ref{thm:main-hyperbolic}}
\subsection{Application of the $\bar\partial$-estimate} 
We begin by applying Theorem~\ref{thm:AHMBerezin} in our setting to get the following theorem, 
which gives the crucial $L^2$-control of solutions to $\bar\partial u=\bar\partial(\chi  f)$ on the entire disk. 
We let $0<r<1$ and denote by $\chi$ a cut-off function $\chi_{\delta,r}$, as in \eqref{eqn:cut-off}.
\begin{thm}\label{thm:d-bar}
Let $0<r<1$ and let $f$ be a bounded holomorphic function on $\D$. 
There exists a solution $u=u_{0,r}$ to $\bar\partial u=\bar\partial(\chi  f)$ that enjoys the estimate
$$
\int_{\D}\lvert u(z) \rvert^2(1-\lvert z\rvert^2)dA(z)\leq \int_{\A((1-\delta)r,r)}\lvert \bar\partial \chi(z)\rvert^2\lvert f(z)\rvert^2(1-\lvert z\rvert^2)^3dA(z).
$$
Moreover,  
$$
\lvert u(z)\rvert=O\left(\lvert z\rvert^{n-1}\right), \qquad z\to\infty
$$
where $n=n(r)=\lceil r^2(1-r^2)^{-1}\rceil$.
\end{thm} 
\begin{proof}
Let $\phi(z)=\log{1/(1-\lvert z\rvert^2)}$ for $z\in\D$, and extend it to the entire plane by defining $\phi(z)=+\infty$ for $z\in\overline{\D^e}$.
The compact set $\mathscr{T}$ is taken to be $\overline{\D(0,r)}$, for a fixed $0<r<1$. 

Define the function $\widehat\phi$ as the minimal subharmonic function of class $\mathscr{C}^{1,1}$, that agrees with $\phi$ on $\D(0,r)$. 
Since $\phi$ is radial, the function $\widehat{\phi}$ is readily found; indeed since 
$$
\phi_{\vert\partial\D(0,r)}=\log \frac{1}{1-r^2},\qquad \partial_n\phi_{\vert \partial\D(0,r)}=\frac{2r}{1-r^2},
$$
one easily checks that
$$
\widehat{\phi}(z):=\frac{r^2}{1-r^2}\log\frac{\lvert z\rvert^2}{r^2}+\log\frac{1}{1-r^2}
$$
is a candidate, in that it agrees with $\phi$ in the right sense. 
Since it is harmonic in the exterior disk $\D(0,r)^e$, 
it follows by the maximum principle for subharmonic functions that it is the correct choice.
With this pair $(\phi,\widehat\phi)$, all assumptions of Theorem~\ref{thm:AHMBerezin} are satisfied.

Applying the theorem, we obtain a solution $u\in L^2_{n,\phi}$ for which the estimate
$$
\int_{\C}\lvert u\rvert^2e^{-\phi}dA\leq \int_{\mathscr{T}}\lvert \bar\partial(\chi f)\rvert^2\frac{e^{-\phi}}{\Delta\widehat\phi},
$$
holds true. Since $\Delta\widehat{\phi}=(1-\lvert z\rvert^2)^{-2}$ on $\mathscr{T}=\D(0,r)$, 
and since $f$ is holomorphic, it follows that
$$
\int_{\D}\lvert u(z)\rvert^2(1-\lvert z\rvert^2)dA\leq\int_{\A((1-\delta)r,r)}\lvert \bar\partial\chi \rvert^2\lvert f\rvert^2(1-\lvert z\rvert^2)^3dA,
$$
which completes the proof.
\end{proof}

\subsection{Non-concentration estimates for minimizers} 
Just as in the planar case, it is clear that for each $0 <r<1$ there exists a holomorphic function $f_0$ which 
attains the value $\inf_f\rHr(f)$, taken over all polynomials. 
The estimates of $L^2_{n,\phi}$-minimal solutions to $\bar\partial u=\bar\partial(\chi f)$, 
where $f=f_r$ is such a minimizer, control the norm of $u$ in terms of the behaviour of $f$ near $\partial\D$, as $r\to 1^-$. 
In this section we aim to understand this behaviour of $f$ better. We begin with the following simple variational identity.
\begin{lemma}\label{lem:unif-norm}
Let $f=f_r$ be a minimizer of $\rho_{\H,r}(f)$. Let
$$
\ell_{k,r}:=\frac{1}{\log\frac{1}{1-r^2}}\int_{\D(0,r)}\lvert f(z)\rvert^k (1-\lvert z\rvert)^{k-1}dA(z),\quad k\in\{1,2\}, \quad 0<r<1.
$$
Then $\ell_{1,r}=\ell_{2,r}$, and both sequences are bounded.
\end{lemma}
\begin{proof}
Consider the variation
$$
V(\alpha)=\rho_{\,\H, r}(\alpha f),\qquad 0<\alpha<\infty.
$$
Since $\alpha f$ is admissible for any $\alpha$, it follows that if $f$ is a minimizer then $V'(1)=0$.
By expanding the square, one observes that this says that $\ell_{2,r}=\ell_{1,r}$. Calculating $\rho_{\,\H,r}(f)$ using this equality, we see that
$$
\rho_{\,\H,r}(f)=1-\ell_{1,r}.
$$
Since $0<\rho_{\,\H,r}(f)\leq 1$ and since the integrals are positive, it follows that $\ell_{1,r}$ is uniformly bounded, and thus the same holds for $\ell_{2,r}$.
\end{proof}
In the following, we will consider limiting proceedures as $r\to 1^-$. 
Many objects will depend on $r$, and sometimes on subsequences $\mathcal{R}=\{r_k\}_{k\geq 1}$. 
In order not to obscure the notation, we will often suppress indices when no confusion should occur. 
\begin{lemma}\label{lem:L1-non-concentration}
Let $\delta$ be a sequence tending to zero as $r\to 1^-$ and denote by $f$ a minimizer of $\rHr(f)$. We have that
$$
\frac{1}{\log\frac{1}{1-r^2}}\int_{\A(r(1-\delta),r)}\lvert f(z)\rvert dA(z)=O\left(\left(1-\frac{\log\frac{1}{1-r^2(1-\delta)^2}}{\log\frac{1}{1-r^2}}\right)^{1/2}\right).
$$
In particular, if $\delta = (1-r)$, the $O$-expression is $o(1)$.
\end{lemma}
\begin{proof}
Cauchy--Schwarz inequality gives that
$$
\int\limits_{\A(r(1-\delta),r)}\lvert f(z)\rvert dA(z)\leq \left(\,\int\limits_{\A(r(1-\delta),r)}\lvert f(z)\rvert^2(1-\lvert z\rvert^2)dA(z)\right)^{1/2}\left(\,\int\limits_{\A(r(1-\delta),r)}\frac{dA(z)}{1-\lvert z\rvert^2}\right)^{1/2}.
$$
Using that $(\log\frac{1}{1-r^2})^{-1}\int_{\D(0,r)}\lvert f\rvert^2(1-\lvert z\rvert^2)dA$ is bounded, 
we may estimate further
$$
\frac{1}{\log\frac{1}{1-r^2}}\int\limits_{\A(r(1-\delta),r)}\lvert f(z)\rvert dA(z)\leq C\left(\frac{\int_{\A(r(1-\delta),r)}\frac{dA}{1-\lvert z\rvert^2}}{\log\frac{1}{1-r^2}}\right)^{1/2}.
$$
Calculating the integrals, we find that
$$
\frac{\int_{\A(r(1-\delta),r)}\frac{dA}{1-\lvert z\rvert^2}}{\log\frac{1}{1-r^2}}=\frac{\log\frac{1}{1-r^2}-\log\frac{1}{1-r^2(1-\delta)^2}}{\log\frac{1}{1-r^2}}=1-\frac{\log\frac{1}{1-r^2(1-\delta)^2}}{\log\frac{1}{1-r^2}},
$$
which proves the first assertion. 

Next, if $\delta=(1-r)$, then we note that
$$
1-\frac{\log(1-r^2(1-\delta)^2)}{\log(1-r^2)}=1-\frac{\log(1-r+r(1-r))+O(1)}{\log(1-r)+O(1)}=1-\frac{\log(1-r)+O(1)}{\log(1-r)+O(1)}=o(1),
$$ 
which completes the proof.
\end{proof}
Consider the functional $\rho_{\,\H,r,\alpha}(f)$, defined for $\alpha<1$ by
$$
\rho_{\,\H,r,\alpha}(f)=\frac{\alpha^2}{\log\frac{1}{1-r^2}}\int_{\D(0,r)}\left((1-\lvert \alpha z\rvert^2)\lvert f(z)\rvert-1\right)^2\frac{dA(z)}{1-\lvert \alpha z\rvert^2},
$$
We have the following lemma, allowing for the freedom of an extra parameter.
\begin{prop}\label{prop:parameter-freedom} 
Let $\alpha_r$ be a sequence of numbers $\alpha_r\to1^{-}$, such that $\alpha_r\geq r^{k}$ for some $k$. Then
$$
\liminf_{r\to 1^{-}}\inf_f\rho_{\,\H, r,\alpha}(f)\geq \rH.
$$
\end{prop}
The condition $\alpha\geq r^{k}$ is by no means meant to be sharp. 
It illustrates some flexibility compared to the restrictions on $\delta$, 
while it is clearily compatible with $\alpha=r$, which corresponds to the choice $\delta=1-r$ which will be made shortly.
\begin{proof}
We have that
$$
\rho_{\,\H,r, \alpha}(f)=\frac{\alpha^2}{\log\frac{1}{1-r^2}}\int_{\D(0,r)}\left((1-\lvert \alpha z\rvert^2)\lvert f(z)\rvert-1\right)\frac{dA(z)}{1-\lvert \alpha z\rvert^2}=\frac{\log\frac{1}{1-\alpha^2r^2}}{\log\frac{1}{1-r^2}}\rho_{\,\H,\alpha r,1}\left(f\left(\frac{z}{\alpha}\right)\right).
$$
If $\alpha\geq r^{k}$, it follows that
$$
1\geq \frac{\log\frac{1}{1-\alpha^2r^2}}{\log\frac{1}{1-r^2}}\geq\frac{\log \frac{1}{1-r^{k+1}}+O(1)}{\log\frac{1}{1-r}+O(1)}=1+o(1).
$$
Since $\alpha r<1$ and $f_{\alpha^{-1}}(z)=f(z/\alpha)$ is admissible, the result follows.
\end{proof}
The following theorem is the key ingredient to the proof of our main result, and will be referred to as the hyperbolic $L^2$-non-concentration estimate.
\begin{thm}\label{thm:L2-non-conentration}
Let $f$ be a minimizer of $\rHr(f)$, 
and let $\delta=(1-r)$. Then 
$$
\frac{1}{\log\frac{1}{1-r^2}}\int_{\A(r(1-\delta),r)}\lvert f(z)\rvert^2(1-\lvert z\rvert^2)dA(z)=o(1)
$$
as $r\to 1^-$ along a sequence $\mathcal{R}=\{r_k\}$ for which $\inf_f\rho_{\,\H, r_{k}}(f)\to\rH$.
\end{thm}
\begin{proof}
Let $\mathcal{R}$ be a sequence of indices along which $\inf_f\rHr(f)\to\rH$. Denote by $c_0$ the number
$$
c_0=\limsup_{r\to1^-, r\in\mathcal{R}}\frac{1}{\log\frac{1}{1-r^2}}\int_{\A(r(1-\delta),r)}\lvert f(z)\rvert^2(1-\lvert z\rvert^2)dA(z).
$$
Let $\alpha=1-\delta$. Consider the functions $g(z)=f(\alpha z)$, and the functionals $\rho_{\,\H, r, \alpha}(g)$. These satisfy
\begin{align}
\rho_{\,\H, r, \alpha}(g) &=\frac{\alpha^2}{\log\frac{1}{1-r^2}}\int_{\D(0,r)}\left((1-\lvert \alpha z\rvert^2)\lvert f(\alpha z)\rvert - 1\right)^2\frac{dA(z)}{1-\lvert \alpha z\rvert^2}  \\
& =\frac{1}{\log\frac{1}{1-r^2}}\int_{\D(0,(1-\delta)r)}\left((1-\lvert z\rvert^2)\lvert f(z)\rvert - 1\right)^2\frac{dA(z)}{1-\lvert z\rvert^2}.
\end{align}\label{eqn:L2-est-1}
By Proposition~\ref{prop:parameter-freedom} it follows that $\liminf_{r\to 1^-, r\in\mathcal{R}}\rho_{\,\H, r, \alpha}(g)\geq \rH$. However, computing the functionals by expanding the squares, we see that
\begin{align*}
\rho_{\,\H,r,\alpha}(g) &=\rho_{\,\H, r}(f)-\frac{1}{\log\frac{1}{1-r^2}}\left(\int_{\A((1-\delta)r,r)}\lvert f(z)\rvert^2(1-\lvert z\rvert^2)dA(z)-2\int_{\A((1-\delta)r,r)}\lvert f(z)\rvert dA(z)\right) \\
& = \rho_{\,\H, r}(f)-\frac{1}{\log\frac{1}{1-r^2}}\int_{\A((1-\delta)r,r)}\lvert f(z)\rvert^2(1-\lvert z\rvert^2)dA(z)+o(1),
\end{align*}
where the last equality follows from Lemma~\ref{lem:L1-non-concentration}.
Taking the lower limit in \eqref{eqn:L2-est-1} as $r\to 1^-$ along $\mathcal{R}$, we find that
$$
\rH\leq \liminf_{r\to 1^-}\rho_{\,\H, (1-\delta)r}(f)=\rH-c_0.
$$
Since $c_0\geq 0$, clearily it follows that $c_0=0$.
\end{proof}
\begin{prop}\label{prop:u-est}
Let $f$ be a minimizer of $\rHr(f)$ and let $\delta=1-r$. 
Let $u$ be the solution to $\bar\partial u=\bar\partial(\chi f)$ from Theorem~\ref{thm:d-bar}. Then 
$$
\frac{1}{\log\frac{1}{1-r^2}}\int_{\D}\lvert u_r\rvert^2(1-\lvert z\rvert^2)dA(z)=o(1)
$$
as $r\to 1^{-}$ along a sequence $\mathcal{R}$ along which $\inf_f\rHr(f)\to\rH$.
\end{prop}
\begin{proof}
From Theorem~\ref{thm:d-bar} we have the estimate
$$
\frac{1}{\log\frac{1}{1-r^2}}\int_{\D}\lvert u(z)\rvert^2(1-\lvert z\rvert^2)dA(z)\leq \frac{1}{\log\frac{1}{1-r^2}}\int_{\A((1-\delta)r,r)}\lvert \bar\partial \chi\rvert^2\lvert f(z)\rvert^2(1-\lvert z\rvert^2)^3dA(z).
$$
We estimate this by
$$
\frac{1}{\log\frac{1}{1-r^2}}\int_{\A((1-\delta)r,r)}\lvert \bar\partial \chi\rvert^2\lvert f(z)\rvert^2(1-\lvert z\rvert^2)^3dA(z)\leq \frac{\left\lVert\phantom{\huge\vert} \bar\partial\chi \, (1-\lvert z\rvert^2)\;\right\rVert_\infty^2}{\log\frac{1}{1-r^2}}\int_{\A((1-\delta)r,r)}\lvert f(z)\rvert^2(1-\lvert z\rvert^2)dA(z).
$$
The supremum norm may be estimated by
$$
\lvert \bar\partial\chi(z)\rvert^2 (1-\lvert z\rvert^2)^2\leq C\frac{(1-r^2(1-\delta)^2)^2}{\delta^2}=O\left(\left(\frac{1-r+r\delta}{\delta}\right)^2\right).
$$
Since $\delta=1-r$, the latter expression is $O(1)$. 
Invoking Theorem~\ref{thm:L2-non-conentration} and using again that $\delta=1-r$ completes the proof.
\end{proof}
\begin{rem}
In order to control $\lVert \bar\partial \chi\, (1-\lvert z\rvert^2)\rVert_\infty$, the parameter $\delta$ needs to be controlled from below, to avoid $\chi$ dropping off too steeply.
On the other hand, in order to apply Theorem~\ref{thm:L2-non-conentration} we need instead an upper bound on the same quantity. The choice $\delta=1-r$ balances these very well (but is probably not sharp). 
\end{rem}
\subsection{Proof of the Main Theorem}
Let $r\to 1^-$, along a subsequence $\mathcal{R}=\{r_k\}$ such that there are admissible $f_k$ for which $\rho_{\,\H, r_k}(f_k)\to\rH$. Let $\delta=(1-r)$, ensuring that all estimates from the previous results come into play. 
Let $f$ be minimizers of $\rho_{\H, r}(f)$, and let $u$ be the $L^2_{n(r),\phi}$-minimal solutions to
$\bar\partial u=\bar\partial(\chi f)$. Put $\nu=\chi  f-u$. 
Then by Liouville's Theorem, $\nu$ is a polynomial of degree at most $n= \lceil r^2/(1-r^2)\rceil$.
As such, it is admissible for $\rHr^*$. Calculating this functional, we obtain
$$
\rho^*_{\H,r}(\nu)=\rHr(\nu)+\frac{1}{\log\frac{1}{1-r^2}}\int_{\D\setminus\D(0,r)}\lvert u(z)\rvert^2(1-\lvert z\rvert^2)dA(z)=:\rHr(\nu)+I_{\text{ext}}.
$$
By Proposition~\ref{prop:u-est}, the term $I_{\text{ext}}$ is $o(1)$, 
so disappears as we take the lower limit. 

We focus on the term $\rHr(\nu)$.
Expanding the square, we find that
$$
\rHr(\nu)=\frac{1}{\log\frac{1}{1-r^2}}\int_{\D(0,r)}\lvert \nu\rvert^2(1-\lvert z\rvert^2)dA-\frac{2}{\log\frac{1}{1-r^2}}\int_{\D(0,r)}\lvert \nu\rvert dA+1=:I_{L^2}-2I_{L^1}+1.
$$
Turning first to $I_{L^2}$, we see that
$$
I_{L^2}=\frac{1}{\log\frac{1}{1-r^2}}\int_{\D(0,r)}\left(\lvert \chi f\rvert^2-2\Re[\chi f \overline{u}]+\lvert u\rvert^2\right)(1-\lvert z\rvert^2)dA.
$$
We estimate the three terms separately: the main contribution comes from the first term;
$$
\frac{1}{\log\frac{1}{1-r^2}}\left\lvert \int_{\D(0,r)}(\lvert \chi_\delta f\rvert^2(1-\lvert z\rvert^2)dA-\int_{\D(0,r)}\lvert f\rvert^2(1-\lvert z\rvert^2)dA\right\rvert=o(1)
$$
by the $L^2$-non-concentration estimate. The middle term is handled as follows: $\Re[\chi f\overline{u}]\leq \lvert \chi f\rvert \,\lvert u\rvert$, and
\begin{align*}
\frac{1}{\log\frac{1}{1-r^2}}\int_{\D(0,r)}\lvert \chi f\rvert \, \lvert u\rvert(1-\lvert z\rvert^2)dA(z) &\leq \left(\frac{1}{\log\frac{1}{1-r^2}}\int_{\D(0,r)}\lvert f\rvert^2(1-\lvert z\rvert^2)dA\right)^{1/2} \\
&\times \left(\frac{1}{\log\frac{1}{1-r^2}}\int_{\D(0,r)}\lvert u\rvert^2(1-\lvert z\rvert^2)dA\right)^{1/2}.
\end{align*}
Since the $L^2$-norm of $f$ is bounded independently of $r$, it follows by applying Proposition~\ref{prop:u-est} to the second factor that the expression is $o(1)$.
The third term is also $o(1)$, in light of the $\bar\partial$-estimate Proposition~\ref{prop:u-est}. In summary: 
$$
I_{L^2}=\int_{\D(0,r)}\lvert f\rvert^2(1-\lvert z\rvert^2)dA+o(1).
$$
Next, turning to $I_{L^1}$ we find that
$$
\left\lvert I_{L^1}-\frac{1}{\log\frac{1}{1-r^2}}\int_{\D(0,r)}\lvert f\rvert dA\right\rvert\leq \frac{1}{\log\frac{1}{1-r^2}}\left(2\int_{\A(r(1-\delta), r)}\lvert f\rvert dA+\int_{\D(0,r)}\lvert u\rvert dA\right).
$$
The first term on the right is $o(1)$ by Proposition~\ref{prop:u-est}, 
and, using Cauchy--Schwarz inequality and Proposition~\ref{prop:u-est} we find that the second term also vanishes in the limit. It follows that
$$
I_{L^1}=\frac{1}{\log\frac{1}{1-r^2}}\int_{\D(0,r)}\lvert f\rvert dA+o(1).
$$
Thus, considering the sequence $\mathcal{R}=\{r_k\}$ along which $\rho_{\,\H,r}(f)$ tends to $\rH$, we may write
\begin{align*}
\rho_{\,\H,r_k}^*(\nu)&=I_{L^2}-2I_{L^1}+1+I_{\text{ext}}=\frac{1}{\log\frac{1}{1-r^2}}\int_{\D(0,r)}\lvert f\rvert^2(1-\lvert z\rvert^2)dA-\frac{2}{\log\frac{1}{1-r^2}}\int_{\D(0,r)}\lvert f\rvert + 1 + o(1)\\ 
&=\rho_{\,\H,r}(f)+o(1).
\end{align*}
It thus follows that
$$
\rH^*\leq\lim_{k\to\infty}\rho_{\,\H,r_k}^*(\nu_{r_k})=\rH.
$$
Put together with the trivial inequality $\rH^*\geq\rH$, this concludes the proof. 
\qed

\begin{proof}[Proof of Corollaries~\ref{cor:degree-hyperbolic} and \ref{cor:degree-planar}]
We begin with Corollary~\ref{cor:degree-hyperbolic}.
Let $\mathcal{R}$ be the subsequence of the previous proof. The statement regarding $\rH^*$ follows immediately from the fact that the polynomials $\nu_r$ used in the proof of Theorem~\ref{thm:main-hyperbolic} are elements of $\Pol_{n(r)}(\C)$. 
However, since $\rHr(\nu)=\rHr^*(\nu)+o(1)$ as $r\to 1^-$ with $r\in\mathcal{R}$, the result follows for $\rH$ as well.

The analogous planar result, Corollary~\ref{cor:degree-planar}, follows in exactly the same fashion.
\end{proof}
\bibliography{References}{}
\bibliographystyle{amsplain}

\end{document}